\documentclass[a4paper,11pt]{article}
\usepackage[top=3.0cm, bottom=3.0cm, inner=3.0cm, outer=3.0cm, includefoot]{geometry}

\usepackage[utf8]{inputenc}
\usepackage[T1]{fontenc}
\usepackage{authblk}
\usepackage{verbatim}
\usepackage{multirow}
\usepackage{multicol}
\usepackage{hyperref}
\usepackage{amssymb}
\usepackage{amsmath}
\usepackage{mathtools,bbm}
\usepackage{graphicx,tikz}
\usepackage{amsthm}
\usepackage{caption,subcaption}
\usepackage{float}
\usepackage{color}
\usepackage{enumerate}
\usepackage{algorithm,float}
\usepackage{algorithmicx}
\usepackage{algpseudocode}
\usepackage{cancel}

\newcommand{\eChar}{\begin{enumerate}[(i)]}
\newcommand{\eCharR}{\begin{enumerate}[(a)]}
\newcommand{\eBr}{\begin{enumerate}[(1)]}

\newcommand{\Abstract}

\theoremstyle{plain}

\newtheorem{lemma}{Lemma}[section]
\newtheorem{theorem}[lemma]{Theorem}

\newtheorem{corollary}[lemma]{Corollary}
\theoremstyle{definition}

\newtheorem{conjecture}[lemma]{Conjecture}

\newtheorem{remark}[lemma]{Remark}

\newtheorem{problem}[lemma]{Problem}

\makeatletter
\newenvironment{breakablealgorithm}
  {% \begin{breakablealgorithm}
   \begin{center}
     \refstepcounter{algorithm}% New algorithm
     \hrule height.8pt depth0pt \kern2pt% \@fs@pre for \@fs@ruled
     \renewcommand{\caption}[2][\relax]{% Make a new \caption
       {\raggedright\textbf{\ALG@name~\thealgorithm} ##2\par}%
       \ifx\relax##1\relax % #1 is \relax
         \addcontentsline{loa}{algorithm}{\protect\numberline{\thealgorithm}##2}%
       \else % #1 is not \relax
         \addcontentsline{loa}{algorithm}{\protect\numberline{\thealgorithm}##1}%
       \fi
       \kern2pt\hrule\kern2pt
     }
  }{% \end{breakablealgorithm}
     \kern2pt\hrule\relax% \@fs@post for \@fs@ruled
   \end{center}
  }
\makeatother

\def\Cr{{\rm cr}}

\def\l{\left(}
\def\r{\right)}

\title{On a conjecture of Pach-Spencer-T\'oth for graph crossing numbers}

\author{
Kaizhe Chen\thanks{School of the Gifted Young, University of Science and Technology of China, Hefei, Anhui 230026, China. Email: ckz22000259@mail.ustc.edu.cn.}
~~~~
Jie Ma\thanks{School of Mathematical Sciences, University of Science and Technology of China, Hefei, Anhui 230026, China, and Yau Mathematical Sciences Center, Tsinghua University, Beijing 100084, China.
Research supported by National Key Research and Development Program of China 2023YFA1010201 and National Natural Science Foundation of China grant 12125106. Email: jiema@ustc.edu.cn.}
}

\date{\today}

\begin{document}

\maketitle
\thispagestyle{plain}
\begin{abstract}
The crossing number of a graph $G$ denotes the minimum number of crossings in any planar drawing of $G$. 
In this short note, we confirm a long-standing conjecture posed by Pach, Spencer, and T\'oth over 25 years ago, establishing an optimal lower bound on the crossing number of graphs that satisfy some monotone properties. 
Furthermore, we address a related open problem introduced by Pach and T\'oth in 2000, which explores the interplay between the crossing number of a graph, its degree sequence, and its bisection width.
\end{abstract}

%%%%%%%%%%%%%%%%%%%%%%%%%%%%%%%%%%%%%%%%%%%%%%%%%%%%%
\section{Introduction}
Let $G$ be a (simple) graph. An {\it arc} is the image of a continuous injective map $[0,1]\rightarrow {\textbf{R}}^2$. 
A {\it planar drawing} of $G$ is a mapping $f$ that assigns to each vertex $v$ of $G$ a point $f(v)$ in the plane and to each edge $xy$ of $G$ an arc connecting $f(x)$ and $f(y)$, not passing through the image of any other vertex. 
We assume that no three edges have an interior point in common.
The {\it crossing number} $\Cr(G)$ of $G$ denotes the minimum number of crossings in any planar drawing of $G$.
Throughout this paper, for a graph $G$, we let $n(G)$ denote the number of vertices in $G$ and let $e(G)$ denote the number of edges in $G$. 

The study of crossing numbers lies at the nexus of discrete geometry and algorithmic design, with applications ranging from geometric embeddings to VLSI layout optimization. 
Investigating crossing numbers and their variants, such as the rectilinear crossing number, reveals their deep connections to combinatorial geometry, hardness of approximation, and efficient algorithms, making them central to both foundational and applied research in discrete and computational geometry (see \cite{Sch} for a comprehensive survey).
The following famous result, known as {\it the crossing lemma}, which is proved by Ajtai, Chv\'atal, Newborn and Szemer\'edi \cite{ACNS} and independently by Leighton \cite{L}, gives a general lower bound for the crossing number of graphs with given numbers of vertices and edges:
Every $n$-vertex graph $G$ with $e\geq 4n$ edges satisfies 
\begin{equation}\label{equ:cr}
\Cr(G)\ge \frac{1}{64} \frac{e^3}{n^2}.  
\end{equation}
Improvements on the constants can be found in \cite{PT,A}.
Another influential result proved by Garey and Johnson \cite{GJ} states that the crossing number problem is NP-complete.

Answering a question of Simonovits, 
Pach, Spencer and T\'oth \cite{PST} proved that the general lower bound \eqref{equ:cr} can be improved substantially for graphs satisfying certain monotone properties, as follows. 
Let $G$ be a graph with $n$ vertices and $e$ edges. Suppose that there are constants $A, \alpha > 0$ such that  
\begin{equation}\label{equ:e(H)}
\mbox{ any subgraph $H$ of $G$ satisfies } e(H)\le A\cdot (n(H))^{1+\alpha}.
\end{equation}
Then there exist constants $c,c'>0$ depending only on $A$ and $\alpha$ such that
\begin{equation}\label{equ:logn^2}
\mbox{ if } e \ge cn \log^2 n, \mbox{ then the crossing number of } G \mbox{ satisfies } \Cr(G)\ge c'\frac{e^{2+1/\alpha}}{n^{1+1/\alpha}}.
\end{equation}
This bound is tight up to a constant factor as shown in \cite{PST}.
Furthermore, Pach, Spencer and T\'oth \cite{PST} conjectured that the same statement holds even for all $n$-vertex graphs with $e\ge  cn$ edges for a suitable constant $c> 0$ (see also \cite{RS}). 
\begin{conjecture}[Pach-Spencer-T\'oth, \cite{PST}]\label{conj}
    Let $G$ be a graph with $n$ vertices and $e$ edges satisfying \eqref{equ:e(H)} for some constants $A,\alpha>0$. Then there exist constants $c,c'>0$ depending only on $A$ and $\alpha$ such that if $e \geq cn$, then $\Cr(G)\ge c'\frac{e^{2+1/\alpha}}{n^{1+1/\alpha}}.$
\end{conjecture}

The authors of \cite{PST} also verified this conjecture for several interesting monotone properties, including the family of $K_{s,t}$-free graphs as well as the family of $C_{2k}$-free graphs for $k\in\{2,3\}$.
F\"uredi and K\"undgen \cite{FK} obtained an improvement on \eqref{equ:logn^2} by showing that 
in the case $\alpha\in (0,1/2)$, the same bound holds under the weaker condition $e \ge cn \log n$. 

In this paper, we resolve Conjecture~\ref{conj} by proving the following theorem. 
Our argument refines the approach developed in \cite{PST,FK}.

\begin{theorem}\label{main}
    Let $G$ be a graph with $n$ vertices and $e$ edges. Suppose that there are constants $A, \alpha > 0$ such that any subgraph $H$ of $G$ satisfies $e(H)\le A\l n(H)\r^{1+\alpha}.$
    Then there exist constants $c,c'>0$ depending only on $A$ and $\alpha$ such that
$$\mbox{ if } e \ge cn, \mbox{ then the crossing number of } G \mbox{ satisfies } \Cr(G)\ge c'\frac{e^{2+1/\alpha}}{n^{1+1/\alpha}}.$$
\end{theorem}

The following result is a direct consequence of Theorem \ref{main}, where the cases of $k\in \{2,3\}$ are first proved in \cite[Theorems 3.1 and 3.2]{PST}.

\begin{corollary}\label{C2k}
    Let $k\ge 2$ be an integer and $G$ be a graph of $n$ vertices and $e$ edges, which contains no cycle of length $2k$. Then there exist two constants $c, c'>0$ depending only on $k$ such that if $e \ge cn$, then the crossing number of $G$ satisfies
    $$\Cr(G)\ge c'\frac{e^{2+k}}{n^{1+k}}.$$
\end{corollary}

The study of crossing numbers is deeply intertwined with other graph parameters, particularly the {\it bisection width}, which has proven instrumental in bridging combinatorial and topological properties of graphs.
Let $G$ be a graph with $n$ vertices. For any bipartition $V(G)=V_1\cup V_2$, let $E(V_1, V_2)$ denote the set of edges in $G$ with one endpoint in $V_1$ and the other endpoint in $V_2$. 
The bisection width, denoted by $b(G)$, of $G$ is defined as
$$b(G)=\min |E(V_1, V_2)|,$$
where the minimum is taken over all bipartitions $V(G) = V_1 \cup V_2$ with $|V_1|,|V_2|\ge n/3$.
Leighton \cite{L83} was the first to identify a strong correlation between the bisection width and the crossing number of a graph, rooted from the Lipton-Tarjan separator theorem \cite{LT} for planar graphs.
This was enhanced in the following remarkable theorem, established by Pach, Shahrokhi and Szegedy \cite{PSS}
and independently by S\'ykora and Vrt'o \cite{SV}:
Let $G$ be a graph of $n$ vertices, whose degrees are $d_1, d_2,..., d_n$. Then it holds that
\begin{equation}\label{equ:b(G)}
b(G)\le 6.32\sqrt{\Cr (G)}+ 1.58\sqrt{\sum_{i=1}^n d_i^2}.
\end{equation}
We would like to emphasize that
the proof of \eqref{equ:logn^2} in \cite{PST} heavily relies on the use of the above theorem. 
The example of a $star$ shows that the inequality \eqref{equ:b(G)} does not remain true if we remove the last term on its right hand side. 
Pach and T\'oth \cite[Problem 5]{PT00} asked whether the dependence of the inequality \eqref{equ:b(G)} on the degrees of the vertices can be improved.

\begin{problem}[Pach-T\'oth, \cite{PT00}, Problem 5]
    Determine the possible values of $t$ such that
    \begin{align}\label{t}
        b(G)=O\left( \sqrt{\Cr (G)}+ \left( \sum_{i=1}^n d_i^t \right)^{1/t} \right)
    \end{align}
    holds for every graph $G$ of $n$ vertices with degrees $d_1, d_2,..., d_n$.
\end{problem}

In the next result, we solve this problem by determining all possible values of $t$ for which \eqref{t} holds. 

\begin{theorem}\label{problem}
    Let $t$ be a positive constant. For $0<t\le 2$, every graph $G$ of $n$ vertices with degrees $d_1, d_2,..., d_n$ satisfies
    \begin{align}\label{t2}
        b(G)= O \l \sqrt{\Cr (G)}+ \left( \sum_{i=1}^n d_i^t \right)^{1/t} \r.
    \end{align}
    For $t>2$, there exist infinitely many integers $n$ with a graph $G$ of $n$ vertices having degrees $d_1, d_2,..., d_{n}$ such that 
    \begin{align}\label{t3}
        b(G)\ge \frac{1}{12} n^{1/2-1/t} \cdot \left( \sqrt{\Cr (G)}+ \left( \sum_{i=1}^{n} d_i^t \right)^{1/t} \right).
    \end{align}
    Therefore, \eqref{t} holds if and only if  $0<t\le 2$.
\end{theorem}

Finally, we would like to present the following result as a dual theorem of Theorem \ref{main}, 
demonstrating the possibility to bound the number of edges in a graph through estimates on the crossing numbers.

\begin{theorem}\label{dual}
    Let $G$ be a graph. Suppose that there are constants $N, \alpha > 0$ such that any subgraph $H$ of $G$ with at least $N$ edges satisfies
    $$\Cr(H)\le \frac{\l e(H) \r^2}{2^{16+3/\alpha}}.$$
    Then there exists a constant $A$ which depends only on $N$ and $\alpha$ such that 
    $$e(G)\le A\cdot \l n(G)\r^{1+\alpha}.$$
\end{theorem}

We hope that Theorem \ref{dual} would facilitate the solution of other unsolved problems in combinatorial and computational geometry, such as the unit distance problem and the $k$-sets problem (see e.g. \cite{M}). 
For instance, if one can prove that the crossing number of any unit distance graph $G$ of $n$ vertices and $e$ edges satisfies $\Cr (G)= O(e^2/ \log\log e)$, then Theorem \ref{dual} implies that $e=n^{1+o(1)}$. 

The rest of the paper is organized as follows. 
In Section \ref{Crossing number and bisection width}, we prove Theorem \ref{problem}. 
In Section \ref{proof}, we prove Theorem \ref{main}, from which we derive Corollary \ref{C2k} and Theorem \ref{dual}.

\section{Crossing number and bisection width}\label{Crossing number and bisection width}

\begin{proof}[Proof of Theorem \ref{problem}]
    First, we consider the case that $0<t\le 2$. For any $x\ge 0$, let $f(x)=x^{2/t}$. Then the second derivative of $f(x)$ is
    $$f''(x)=\frac{2}{t}\l \frac{2}{t}-1 \r x^{2/t-2}\ge 0.$$
    For $1\le i\le n$, let $x_i=d_i^t$. It follows by the Jensen's inequality that
    $$\sum_{i=1}^n f(x_i)\le (n-1)f(0)+ f\l \sum_{i=1}^n x_i \r= f\l \sum_{i=1}^n x_i \r.$$
    That is
    $$\left( \sum_{i=1}^n d_i^2 \right)^{1/2} \le \left( \sum_{i=1}^n d_i^t \right)^{1/t}.$$
    Then, the inequality \eqref{t2} follows from \eqref{equ:b(G)}.

    Now suppose that $t>2$. Let $G$ be a graph defined as follows. 
    The vertex set of $G$ is the following subset of integer points in ${\textbf{R}}^2$
    $$V(G)= \left\{ (i,j) | i,j\in [n] \right\},\ {\rm where}\ [n]=\{1,2,...,n\}.$$ 
    Two vertices in $V(G)$ are adjacent if and only if the Euclidean distance between them is one (see Figure \ref{55}). 
    It is clear that $G$ is a planar graph with $n^2$ vertices of maximum degree four. 
    We claim that $b(G)\ge n/3$. 
    Indeed, let $V(G)= V_1\cup V_2$ be any partition of $V(G)$ such that $|V_1|,|V_2|\ge n^2/3$. 
    It suffices to prove that $|E(V_1, V_2)|\ge n/3$.
    \begin{figure}[htbp]
    \centering
    \includegraphics[scale=0.3]{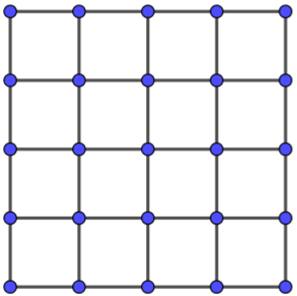}
    \caption{A drawing of $G$ with $n=5$}
    \label{55}
    \end{figure}
    
    Suppose that there is an $i_0\in [n]$ such that $(i_0,j)\in V_1$ for every $j\in [n]$. Since $|V_1|\le 2n^2/3$, there are at most $2n/3$ number of $j$ such that $(i,j)\in V_1$ for every $i\in [n]$. Thus, there are at least $n/3$ columns containing at least one edge in $E(V_1, V_2)$, and hence $|E(V_1, V_2)|\ge n/3$. 

    Now, we assume that there does not exist an $i_0\in [n]$ such that $(i_0,j)\in V_1$ for every $j\in [n]$. 
    That is, each row contains at least one vertex in $V_2$.
    Since $|V_1|\ge n^2/3$, there are at least $n/3$ rows containing at least one vertex in $V_1$. Each such row contains at least one edge in $E(V_1, V_2)$, and hence $|E(V_1, V_2)|\ge n/3$.

    Thus, $b(G)\ge n/3$. Since $\Cr (G)=0$, we have
    $$\sqrt{\Cr (G)}+ \left( \sum_{i=1}^{n^2} d_i^t \right)^{1/t} \le \left( \sum_{i=1}^{n^2} 4^t \right)^{1/t}=4n^{2/t}\le 12n^{2/t-1} b(G),$$
    completing the proof.
\end{proof}

\begin{remark}
    In the proof of Theorem \ref{problem}, for any integer $n$, we give a counterexample $G$ of $n^2$ vertices which is planar, i.e., $\Cr(G)=0$. 
    The graph $G$ can be easily turned into counterexamples $G'$ with unbounded $\Cr(G')$, by the following operations:
    choosing any $o(n^2)$ edges of $G$ and 
    replacing each of these edges by a copy of $K_{s,s}$ (for any fixed $s\geq 3$).\footnote{To be more precise, let $S$ be the set of vertices in these $o(n^2)$ edges. 
    Then $G'$ is obtained from $G$ by replacing each vertex $v\in S$ with an independent set $I_v$ of size $s$, adding all edges between vertices in $I_v$ and vertices $u\in V(G)\setminus S$ if $uv\in E(G)$, and adding all edges between $I_v$ and $I_u$ if $uv\in E(G)$.} 
    In this way, we can obtain a non-planar counterexample with $o(n^2)$ number of crossings.
\end{remark}

\section{Proofs of main results}\label{proof}
In this section, we prove Theorem \ref{main}, Corollary \ref{C2k} and Theorem \ref{dual}.

\begin{proof}[Proof of Theorem \ref{main}]
    Let $G$ be a graph with $n$ vertices and $e$ edges. Suppose that there are constants $A, \alpha > 0$ such that, for any subgraph $H$ of $G$, we have
    \begin{align}\label{assumption}
        e(H)\le A(n(H))^{1+\alpha}.
    \end{align}
    To prove Theorem \ref{main}, we assume for a contradiction that $e \ge cn$,  
    $$\Cr(G)< c'\frac{e^{2+1/\alpha}}{n^{1+1/\alpha}},$$
    and $G$ is drawn in the plane with exactly $\Cr(G)$ crossings, where $c$ and $c'$ are suitable constants which will be determined later (see \eqref{equ:cc'}).

    First, we split every vertex of $G$ whose degree exceeds $\overline{d} := 2e/n$ into vertices of degree at most $\overline{d}$, as follows. Let $v$ be a vertex of $G$ such that the degree $d$ of $v$ in $G$ satisfies $d > \overline{d}$. Let $vw_1, vw_2,..., vw_d$ be the edges incident to $v$, listed in clockwise order. Replace $v$ by $\lceil d/\overline{d} \rceil$ new vertices, $A_v=\{v_1, v_2,...,v_{\lceil d/\overline{d} \rceil}\}$, placed in clockwise order on a very small circle $C$ around $v$. 
    Connect $w_j$ to $v_i$ if and only if $\overline{d}(i-1) < j \le \overline{d}i$ for all $1 \le j \le d$ and $1 \le i \le \lceil d/\overline{d} \rceil$. 
    Let the circle $C$ be small enough such that we do not introduce any new crossings. Repeat this procedure until the degree of each vertex is at most $\overline{d}$, and denote the resulting graph by $G'$.

    By the construction of $G'$, we have $e(G')=e(G)=e$, and 
    \begin{align}\label{crG}
        \Cr(G')\le \Cr(G)< c'\frac{e^{2+1/\alpha}}{n^{1+1/\alpha}}.
    \end{align}
    Observe that 
    $$n(G')-n(G)=\sum_{v:d(v)> \overline{d}} (|A_v|-1)< \sum_{v:d(v)> \overline{d}} \frac{d(v)}{\overline{d}}\le \frac{2e}{\overline{d}}=n.$$
    We deduce $N:=n(G')\in [n,2n)$.

    Now, we break $G'$ into smaller components according to the following algorithm:

    \begin{breakablealgorithm}
    \caption{(Decomposition Algorithm)}
    \label{Decomposition Algorithm}
    \begin{algorithmic}[1]
    \item[1:] Let $G^0 = G', G^0_1 = G', M_0 =1,$ and $m_0 = 1$.
    
    \item[2:] For $i\ge 0$, suppose that $G^i$ consists of $M_i$ components $G^i_1$, $G^i_2$, ..., $G^i_{M_i}$, each of at most $(2/3)^i N$ vertices. Let $m_i\in [0, M_i]$ be the constant such that
    \begin{align}\label{mi}
        \l\frac{2}{3}\r^{i+1} N\le n(G^i_j) \le \l\frac{2}{3}\r^{i} N &,\ {\rm for}\ 1\le j \le m_i;\\ \label{Mi} n(G^i_j) < \l\frac{2}{3}\r^{i+1} N &,\ {\rm for}\ m_i< j \le M_i.
    \end{align}
    
    \item[3:] Recall the constants $A, \alpha$ from \eqref{assumption}. If
    $$\l\frac{2}{3}\r^{i}\ge \frac{1}{(2A)^{1/\alpha}}\frac{e^{1/\alpha}}{N^{1+1/\alpha}},$$
    then for every $1\le j \le m_i$, delete $b(G^i_j)$ edges from $G^i_j$ to obtain two induced subgraphs, each of at most $(2/3)n(G^i_j)$ vertices. Let $G^{i+1}$ denote the resulting graph on $N$ vertices. 
    Then each component of $G^{i+1}$ has at most $(2/3)^{i+1} N$ vertices. 
    Return to Step 2. 
    \item[4:] Else STOP.
    \end{algorithmic}
    \end{breakablealgorithm}

    Suppose that the Decomposition Algorithm terminates when $i=k$. If $k=0$, then 
    $$1< \frac{1}{(2A)^{1/\alpha}}\frac{e^{1/\alpha}}{N^{1+1/\alpha}}.$$
    It follows that
    $$e> 2A N^{1+\alpha}\ge 2A n^{1+\alpha},$$
    which is contradictory to the assumption \eqref{assumption}. Thus $k>0$, and 
    \begin{align}\label{k}
        \l\frac{2}{3}\r^{k}<\frac{1}{(2A)^{1/\alpha}}\frac{e^{1/\alpha}}{N^{1+1/\alpha}}\le \l\frac{2}{3}\r^{k-1}.
    \end{align}

    We will first show that $G^k$ contains less than $e/2$ edges. The number of vertices of each component of $G^k$ satisfies 
    \begin{align}\label{nG}
        n(G^k_j) \le \l\frac{2}{3}\r^{k} N< \frac{1}{(2A)^{1/\alpha}}\frac{e^{1/\alpha}}{N^{1+1/\alpha}} N=\l \frac{e}{2AN} \r^{1/\alpha}.
    \end{align}
    Let $V_j$ be the subset of $V(G)$ defined as follows:
    $$V_j=\{ v\in V(G) | v\in V(G^k_j)\ {\rm or}\ A_v\cap V(G^k_j)\ne \emptyset \},$$
    i.e., $V_j$ is the vertex subset consisting of all `pre-images' of $V(G_j^k)$. 
    The construction of $G'$ implies that $|V_j|\le n(G^k_j)$. 
    Let $G[V_j]$ be the subgraph of $G$ induced by $V_j$. 
    By the assumption \eqref{assumption} and the inequality \eqref{nG}, we have 
    \begin{align}\label{eGV}
        e(G[V_j])\le A|V_j|^{1+\alpha}\le A \l n(G^k_j) \r^{1+\alpha}< An(G^k_j)\cdot \frac{e}{2AN}=\frac{e}{2N}n(G^k_j).
    \end{align}
    By the construction of $G'$, each edge in $G^k_j$ has a corresponding edge in $G[V_j]$. Thus, we have 
    $$e(G^k_j) \le e\l G[V_j]\r.$$
    It follows, as desired, that
    $$e(G^k)=\sum_{j=1}^{M_k}e(G^k_j)< \frac{e}{2N}\sum_{j=1}^{M_k}n(G^k_j)=\frac{e}{2}.$$

    To obtain a contradiction, it now suffices to show that we deleted at most $e/2$ edges of $G'$ to obtain $G^k$. For any $0\le i < k$, the inequality \eqref{mi} implies that 
    \begin{align}
        m_i\le \l\frac{3}{2}\r^{i+1}.
    \end{align}
    It follows by the Cauchy-Schwarz inequality and \eqref{crG} that,  
    \begin{align}\label{cr}
        \sum_{j=1}^{m_i}\sqrt{\Cr(G^i_j)}\le \sqrt{m_i\sum_{j=1}^{m_i} \Cr(G^i_j)}\le \sqrt{\l\frac{3}{2}\r^{i+1}\Cr(G')}< \sqrt{\l\frac{3}{2}\r^{i+1}} \sqrt{c'\frac{e^{2+1/\alpha}}{n^{1+1/\alpha}}}.
    \end{align}

    For any vertex $v$ in $G^i$, 
    denote by $d(v, G^i)$ the degree of $v$ in $G^i$. Since $G^i$ is a subgraph of $G'$, we have
    $$\max_{v\in V(G^i)}d(v, G^i)\le \max_{v\in V(G')}d(v, G')\le \overline{d},$$
    and $$\sum_{v\in V(G^i)}d(v, G^i)\le \sum_{v\in V(G')}d(v, G')=2e(G')=2e.$$
    Thus, using the Cauchy-Schwarz inequality again, we deduce that for every $0\leq i<k$
    \begin{align}\label{D2}\notag 
        \sum_{j=1}^{m_i}\sqrt{\sum_{v\in V(G^i_j)}d^2(v, G^i_j)} &\le \sqrt{m_i\sum_{v\in V(G^i)}d^2(v, G^i)} \\ \notag 
        &\le \sqrt{\l\frac{3}{2}\r^{i+1}}\sqrt{\max_{v\in V(G^i)}d(v, G^i) \sum_{v\in V(G^i)}d(v, G^i)} \\ 
        &\le \sqrt{\l\frac{3}{2}\r^{i+1}} \sqrt{\overline{d}\cdot 2e}= \sqrt{\l\frac{3}{2}\r^{i+1}} \cdot \frac{2e}{\sqrt{n}}.
    \end{align}
    Using the theorem \eqref{equ:b(G)} and inequalities \eqref{cr} and \eqref{D2}, the total number $\sigma$ of edges deleted during the Decomposition Algorithm is
    \begin{align}\notag
        \sigma=\sum_{i=0}^{k-1}\sum_{j=1}^{m_i}b(G^i_j) &\le 6.32\sum_{i=0}^{k-1}\sum_{j=1}^{m_i} \sqrt{\Cr (G^i_j)} +1.58\sum_{i=0}^{k-1}\sum_{j=1}^{m_i} \sqrt{\sum_{v\in V(G^i_j)}d^2(v, G^i_j)} \\ \notag 
        &< 6.32\sqrt{c'\frac{e^{2+1/\alpha}}{n^{1+1/\alpha}}} \sum_{i=0}^{k-1}\sqrt{\l\frac{3}{2}\r^{i+1}}+ \frac{3.16e}{\sqrt{n}} \sum_{i=0}^{k-1}\sqrt{\l\frac{3}{2}\r^{i+1}} \\ \notag 
        &=\sqrt{\frac{3}{2}}\frac{\sqrt{(3/2)^k}-1}{\sqrt{3/2}-1}\l 6.32\sqrt{c'\frac{e^{2+1/\alpha}}{n^{1+1/\alpha}}}+ \frac{3.16e}{\sqrt{n}} \r 
    \end{align}
    By the inequality \eqref{k} that $(3/2)^{k-1}\leq \frac{(2A)^{1/\alpha}N^{1+1/\alpha}}{e^{1/\alpha}}$ and the fact that $N<2n$, we have
    \begin{align}\notag
        \sigma &< 6.7\sqrt{\frac{(2A)^{1/\alpha}N^{1+1/\alpha}}{e^{1/\alpha}}} \l 6.32\sqrt{c'\frac{e^{2+1/\alpha}}{n^{1+1/\alpha}}}+ \frac{3.16e}{\sqrt{n}} \r \\ \notag 
        &< 45\sqrt{2^{1+2/\alpha}A^{1/\alpha}c'e^2}+ 22\sqrt{2^{1+2/\alpha}A^{1/\alpha}n^{1/\alpha}e^{2-1/\alpha}}.
    \end{align}
    Each term on the right hand side of the above inequality is at most $e/4$, provided that 
    \begin{align}\label{equ:cc'}
       e > cn, ~~ c'= \frac{1}{(180)^22^{1+2/\alpha}A^{1/\alpha}}
        \mbox{ ~~and~~ } c= (88)^{2\alpha}2^{\alpha+2}A.
    \end{align}
    Hence $\sigma<e/2$, which is a contradiction, completing the proof.
\end{proof}

To prove Corollary \ref{C2k}, we need the following classic theorem of Bondy-Simonovits \cite{BS}.

\begin{theorem}[Bondy-Simonovits, \cite{BS}]\label{BS}
For a fixed integer $k\geq 2$, every $n$-vertex graph $G$ that does not contain any cycle of  length $2k$ has at most $100k\cdot n^{1+1/k}$ edges.
\end{theorem}

\begin{proof}[Proof of Corollary \ref{C2k}]
    It's a direct consequence of Theorem \ref{main} and Theorem \ref{BS}.
\end{proof}

To conclude this section, we give a proof of Theorem \ref{dual}.

\begin{proof}[Proof of Theorem \ref{dual}]
Let $N, \alpha > 0$ be constants.
We say a graph $G$ is {\it $(N,\alpha)$-good}, if any subgraph $H$ of $G$ with at least $N$ edges satisfies
    \begin{align}\label{H}
        \Cr(H)\le \frac{\l e(H) \r^2}{2^{16+3/\alpha}}.
    \end{align}
We point out that if $G$ is $(N,\alpha)$-good, then any subgraph of $G$ is $(N,\alpha)$-good as well.
Let $G$ be an $(N,\alpha)$-good graph with $n$ vertices and $e$ edges. 
We aim to show that $e\leq An^{1+\alpha}$, where $A=\max \{ 88^2 2^{1+3/\alpha}, N \}$. 
    
    Suppose for a contradiction that $e>An^{1+\alpha}$. Without losing generality, we may further assume that $G$ is the minimum counterexample. 
    Consider any proper subgraph $H$ of $G$.
    Since $H$ is also $(N,\alpha)$-good, we have $e(H)\le A\l n(H)\r^{1+\alpha}$.
    This also implies that $e-1\leq An^{1+\alpha}$,
    which further implies that $e\le An^{1+\alpha}+1\leq 2An^{1+\alpha}$.
    Now we derive that any subgraph $H$ of $G$ satisfies $e(H)\leq 2A\l n(H)\r^{1+\alpha}$. 
    Using Theorem \ref{main}, if $e\ge cn$, then 
    \begin{align}\label{cc}
        \Cr(G)\ge c'\frac{e^{2+1/\alpha}}{n^{1+1/\alpha}},
    \end{align}
    where $$c'= \frac{1}{(180)^22^{1+2/\alpha} (2A)^{1/\alpha}} \ {\rm and}\ c= (88)^{2\alpha}2^{\alpha+2}(2A).$$
    Since $n^2\ge e > An^{1+\alpha}\ge An$, we have $n>A$, and hence
    $$e> An^{1+\alpha}> A^{1+\alpha}n\ge (88)^{2\alpha}2^{\alpha+3}An= cn.$$
    Combining the inequalities \eqref{H} and \eqref{cc}, we derive
    \begin{align}\notag
        \frac{e^2}{2^{16+3/\alpha}} \ge \Cr(G)\ge c'\frac{e^{2+1/\alpha}}{n^{1+1/\alpha}}> c'\frac{e^2\l An^{1+\alpha} \r^{1/\alpha}}{n^{1+1/\alpha}}=\frac{e^2}{(180)^22^{1+3/\alpha}}.
    \end{align}
    Simplifying the above inequality results in the incorrect statement $(180)^2>2^{15}$, thereby completing the proof.
\end{proof}

\end{document}